\documentclass{amsart}
\usepackage[french,english]{babel}
\usepackage[left=4cm,right=4cm,top=3cm,bottom=3cm]{geometry}
\usepackage{amsfonts, amsthm, amssymb, amsmath, stmaryrd}
\usepackage{mathrsfs,array}
\usepackage{eucal,color,times,enumerate,accents}
\usepackage{tikz-cd}
\usepackage{url}
\usepackage{bbm}
\usepackage[utf8]{inputenc}
\usepackage{pgfplots}
\pgfplotsset{compat=newest}
\usepackage{lipsum}
\usepackage{graphicx}
\usepackage{rotating}
\usepackage{enumitem}
\usepackage{enumerate}
\usepackage{footmisc}
\usepackage{textcomp}
\usepackage[hidelinks]{hyperref}
\usepackage{cleveref}
\usepackage[left=4cm,right=4cm,top=3cm,bottom=3cm]{geometry}
\usepackage{mathtools}

 \numberwithin{equation}{section}
\newtheorem{thm}[equation]{Theorem}

\newtheorem{prop}[equation]{Proposition}
\newtheorem{cor}[equation]{Corollary}

\theoremstyle{definition}
\newtheorem{defn}[equation]{Definition}
\theoremstyle{definition}
\newtheorem{rem}[equation]{Remark}
\newtheorem{note}[equation]{Notation}
\theoremstyle{definition}
\newtheorem{exa}[equation]{Example}

\newtheorem{set}[equation]{Setup}
\usepackage{tikz}
\usepackage{tikz-cd}
\usepackage{rotating}
\usetikzlibrary{arrows,shapes,positioning}
\usetikzlibrary{decorations.markings}
\tikzstyle arrowstyle=[scale=1]
\usepackage{lipsum}
\usepackage{url}
\usepackage{imakeidx}
\makeindex[title=Index of definitions and notations,intoc]

\newcommand{\supp}{\mathrm{supp}}

\begin{document}
\title{Obstructions for the existence of separating morphisms and totally real pencils}
\author{Matilde Manzaroli}
\begin{abstract}
  \setlength{\parindent}{0pt}
  \par It goes back to Ahlfors that a real algebraic curve $C$ admits a separating morphism $f$ to the complex projective line if and only if the real part of the curve disconnects its complex part, i.e. the curve is \textit{separating}. The degree of such $f$ is bounded from below by the number $l$ of real connected components of $ \mathbb{R} C$. The sharpness of this bound is not a priori clear. We prove that real algebraic separating curves, embedded in some ambient surface and with $l$ bounded in a certain way, do not admit separating morphisms of lowest possible degree. Moreover, this result of non-existence can be applied to show that certain real separating plane curves of degree $d$, do not admit totally real pencils of curves of degree $k$ such that $kd \leq l$. 
  \par
  \medskip 
  R\'ESUM\'E. Il remonte \`a Ahlfors qu'une courbe alg\'ebrique r\'eelle $C$ admet un morphisme s\'eparant $f$ \`a la droite complexe projective si et seulement si la partie r\'eelle de la courbe d\'econnecte sa partie complexe, i.e. la courbe est \textit{s\'eparante}. Le degr\'e d'un tel $f$ est born\'e par le bas par le nombre de composantes connexes r\'eelles de $ \mathbb{R} C$. La nettet\'e de cette borne n'est pas claire a priori. Nous prouvons que les courbes alg\'ebriques r\'eelles s\'eparantes, plong\'ees dans une surface ambiante
 et avec $l$ born\'e d'une certaine mani\`ere, n'admettent pas de morphismes s\'eparants de degr\'e le plus bas que possible. De plus, ce r\'esultat de non-existence peut \^etre appliqu\'e pour montrer que certaines courbes r\'eelles s\'eparantes planes de degr\'e $d$, n'admettent pas de pinceaux de courbes de degr\'e $k$ totalement r\'eels tels que $kd \leq l$. 
 \end{abstract}
\maketitle
\tableofcontents
\section{Introduction}
A \textit{real algebraic variety} is a compact complex algebraic variety $X$ equipped with an anti-holomorphic involution $\sigma:X \rightarrow X$, called \textit{real structure}. The \textit{real part} $\mathbb{R}X$ of $X$ is the set of points fixed by the involution $\sigma$. 

Let $C$ be any non-singular real algebraic curve. Denote by $C_1,\dots,C_l$ the connected components of $\mathbb R C$. Harnack-Klein's inequality \cite{Harn76, Klei73} bounds $l$ by the genus $g$ of $C$ plus one. If $C \setminus \mathbb R C$ is connected, we say that $C$ is of \textit{type II} or \textit{non-separating}, otherwise of \textit{type I} or \textit{separating}. Looking at the real part of the curve and its position with respect to its complexification gives information about $l$ and vice versa. For example, we know that if $l$ equals $g+1$, i.e. $C$ is \textit{maximal}, then $C$ is of type I. Or, if $C$ is of type I then $l$ has the parity of $g+1$. Rokhlin promoted and contributed to this new point of view; an example of Rokhlin's contribution is the introduction and the study of the \textit{complex orientations} of a separating real algebraic curve \cite{Rokh74}. Namely, if $C$ is of type I, the two  halves of $C\setminus \mathbb R C$ induce two opposite orientations on $\mathbb R C$ called \textit{complex orientations} of the curve. Looking at complex orientations of separating real curves embedded in some ambient surface has allowed a change of prospective and a remarkable progress in the study of their topology and a refinement of their classifications. 
\begin{defn}
\label{defn: separating}
We say that a real morphism $f$ from a real algebraic curve $C$ to the complex projective line $\mathbb C \mathbb P^1$ is \textit{separating} if $f^{-1}(\mathbb R \mathbb P^1 )= \mathbb R C$. 
\end{defn}
According to Ahlfors \cite[\S4.2]{Ahlf50}, there exists a separating morphism $f: C \rightarrow \mathbb C \mathbb P^1$ if and only if $C$ is separating. In this paper, we focus on the relation between the topology of a real separating curve and the existence of separating morphisms of given degree.
\subsection{Organisation of the paper}
Before stating the main results of this paper, Theorems \ref{thm: main} and \ref{thm: main_D1}, we present them restricted to the case of real algebraic plane projective curves in Section \ref{subsec: real_proj_plane} (Proposition \ref{prop: plane_main}). In Section \ref{subsec: sep}, we present known results and generalities of separating morphisms. In Section \ref{subsec: obstruction_pencil}, we focus on real plane curves admitting totally real pencils (Corollary \ref{cor: odd_even_plane} and Proposition \ref{prop: no_pencil_plane_generale}). Afterwards, Theorems \ref{thm: main} and \ref{thm: main_D1} are stated in Section \ref{subsec: main_statement} and proved in Section \ref{subsec: main}. Finally, in Section \ref{subsec: exa}, we present some examples, applications of the main results and we prove Proposition \ref{prop: no_pencil_plane_generale}.

 \subsection*{Acknowledgements}
I would like to thank Mario Kummer, Stepan Orevkov and Kris Shaw for their interest in this paper; the referee for useful suggestions and comments; Erwan Brugall\'e and Hannah Markwig for remarks and advice on a preliminary version of the paper. Thanks to Athene Grant.

\subsection{Real plane curves}
\label{subsec: real_proj_plane}
Let us consider a real algebraic separating plane projective curve $C$. The following proposition shows that there is a relation between $\deg C$, the number of connected components of $\mathbb R C$ and the existence of separating morphisms $f: C \rightarrow \mathbb C \mathbb P^1$ of a certain degree.

\begin{prop}[Particular case of Theorems \ref{thm: main} and \ref{thm: main_D1}]
\label{prop: plane_main}
Let $C$ be a non-singular real algebraic plane projective curve of type I and with $l$ real connected components. Assume one of the following:
\begin{enumerate}
	\item The degree of $C$ is $2s+1$, for some $s \in \mathbb Z_{\geq 2}$ and 
	
	$$1- \varepsilon + \frac{s^2-5s+4}{2} < \lfloor \frac{l}{2}\rfloor < \frac{s^2+s-2}{2},$$ 
	
	where $\varepsilon \in \{0,1\}$ such that $l\equiv \varepsilon \mod 2$.
	\item The degree of $C$ is $2s$, for some $s \in \mathbb Z_{\geq3}$ and $$ \max(0,1- \varepsilon + \frac{s^2-7s+10}{2}) <\lfloor \frac{l-1}{2}\rfloor  < \frac{s^2-s-2}{2},$$ 
	where $\varepsilon \in \{0,1\}$ such that $l-1\equiv \varepsilon \mod 2$.
\end{enumerate}
Then $C$ admits no separating morphisms of degree $l$.
\end{prop}

Proposition \ref{prop: plane_main} falls within the line of results relating topology, complex orientations and properties of separating plane curves, such as Rokhlin's complex orientations formula (\cite{Rokh74}, \cite{Mish75}), and such as \cite[Theorem 1.1]{Orev21}, where Orevkov shows that there are finer relations for the numbers which intervene in the complex orientations formula.

As corollaries of Proposition \ref{prop: plane_main}, one can find obstructions for the existence of totally real pencils of curves of a certain degree for given separating plane curves; see Section \ref{subsec: obstruction_pencil}.

\subsection{Generalities of separating morphisms}
\label{subsec: sep}
A separating morphism $f:C \rightarrow \mathbb C \mathbb P^1$ is always unramified once restricted to $\mathbb R C$; see \cite[Theorem 2.19]{KumSha20}. Therefore, the restriction of $f$ to each connected component of $\mathbb R C$ is a covering map of $\mathbb R \mathbb P^1 $. This implies that the degree of a separating morphism is at least as big as the number of connected components of $\mathbb R C$. Actually, the definition of separating morphism is more general and includes real morphisms between any real algebraic varieties of same dimension. In order to a have a general idea of the subject, we refer the interested reader to \cite{KumSha20} and \cite{KumLeTMan22}. In the context of this paper, we only need Definition \ref{defn: separating}.

In \cite{Huis01}, \cite{Gaba06},\cite{CopHui13}, \cite{Copp13}, \cite{KumShaw20}, \cite{Orev217} properties of separating morphisms and their existence are treated. For example, \cite[Theorem 7.1]{Gaba06} states that a genus $g$ real separating curve with $l$ real connected components admits a separating morphism of degree at most $\frac{g+ l+1}{2}$. Later, Coppens, in \cite{Copp13}, constructs, for every value $h$ between $l$, the minimum for the degree of a separating morphism, and $\frac{ g+ l +1}{2}$, a separating curve $C$ of genus $ g$ and with $ l$ real connected components such that $h$ is the smallest possible degree of a separating morphism of $C$. In \cite{KumShaw20}, the authors fix a separating curve $C$ and study all separating morphisms of $C$ as follows. Let $\mathbb R C$ consist of $l$ connected components $C_1, \dots, C_l$. Set $d_i(f) \in \mathbb N$ the degree of the covering map $f|_{C_i}: C_i \rightarrow \mathbb R \mathbb P^1 $ and set $d(f):=(d_1(f), \dots, d_l(f))$. Let us denote by $t:=\sum_{i=1}^ld_i(f)$ the degree of $f$. The set $\text{Sep}(C)$ of all such degree partitions forms a semigroup, called \textit{separating semigroup}, and for all elements $d \in \text{Sep}(C)$, satisfying certain conditions, it is shown that $d+ \mathbb Z^l_{\geq 0}$ is also contained in $\text{Sep}(C)$. So, in certain cases, in order to understand $\text{Sep}(C)$ for a given separating curve $C$, it is important to understand which minimal possible element $\text{Sep}(C)$ contains, where with minimal we mean that $t$ is of minimal possible value. 
 \begin{rem}
Thanks to Harnack-Klein inequality, any real curve of genus $g$ cannot have more than $g+1$ real connected components. By Riemann-Roch theorem, all non-singular real curves with $l=g+1$ real connected components admit a separating morphism of degree $l$. But, whenever $l < g+1$, it is not a priori clear whether a separating morphism of degree $l$ exists. 
 \end{rem}
 
\subsection{Obstruction for the existence of totally real pencils}
In this section, we see how obstructions for the existence of separating morphisms may lead to obstructions for the existence of totally real pencils.
\label{subsec: obstruction_pencil}
\begin{defn}
 	\label{defn: totally_real_pencil}
 	Let $C$ be a non-singular real algebraic plane projective curve of type I. We say that $C$ admits a totally real pencil of curves of degree $k$ if there exists an integer $k$ such that there are $f,g \in \mathbb R [x,y,z]_{k}$ such that $V(f,g) \cap C= \emptyset$ and $V(\lambda f + \mu g) \cap C$ consists of real points only for all $\lambda, \mu \in \mathbb R$ not both zero.
 \end{defn}
 
Kummer and Shaw, in \cite{KumShaw20}, also prove that for all real separating curves embedded in the complex projective plane, there exist infinitely many totally real pencils of curves.
 \begin{thm} (\cite[Theorem 1.6]{KumShaw20})
 \label{thm: KS}
 	Let $C$ be a non-singular real algebraic plane projective curve of type I. Then, there exists a positive integer $k$ such that the curve $C$ admits a totally real pencil of curves of degree $k'$, for all $k' \geq k$.
 \end{thm}
 In \cite[Question 3.6]{KumShaw20}, the authors wonder which may be the minimal possible value of $k$ in Theorem \ref{thm: KS}. An immediate corollary of Proposition \ref{prop: plane_main} gives a lower bound for $k$, as follows.

 \begin{cor}
 \label{cor: odd_even_plane}
 Let $C$ be a non-singular degree $d$ real algebraic plane projective curve of type I and with $l$ real connected components. Assume that $C$ satisfies the hypotheses of Proposition \ref{prop: plane_main}. Then $C$ admits no totally real pencil of curves of degree $k$ such that $kd \leq l$.
  \end{cor}
Corollary \ref{cor: odd_even_plane} gives obstructions for the existence of certain  totally real pencils of curves, but we do not know yet how to compute the minimal value for $k$ of Theorem \ref{thm: KS}. 
 \begin{rem}
 As remarked in \cite[Remark 3.5]{KumShaw20}, according to \cite[\S4.2]{Ahlf50}, real separating plane curves admit separating morphisms (whose degree is a priori unknown). Hence, there exists an integer $k$ such that there are $f,g \in \mathbb R [x,y,z]_{k}$ such that $V(\lambda f + \mu g) \cap C$ consists of real points only for all $\lambda, \mu \in \mathbb R$ not both zero. The difference with Definition \ref{defn: totally_real_pencil} is that $V(f,g) \cap C$ may be non-empty. Analogously, \cite[Theorem 1.7]{Gaba06} implies the existence of separating morphisms of degree between $l$ and $\frac{g+l+1}{2}$ and, once again, this does not imply that a given real separating plane curve of genus $g$ and with $l$ real connected components, admits totally real pencils of curves of a certain degree depending on $l$ and $g$, in the sense of Definition \ref{defn: totally_real_pencil}.
	\end{rem}
 In order to have examples of separating real curves for which Corollary \ref{cor: odd_even_plane} holds, we prove the following.
\begin{prop}
 	\label{prop: no_pencil_plane_generale}
 	For all $d \in \mathbb N_{\geq 5}$ and for all $l \geq \lceil \frac{d}{2} \rceil$ such that $l$ has the parity of $\lceil \frac{d}{2} \rceil$ and is bounded as in Proposition \ref{prop: plane_main}, there exists a non-singular real plane projective curve  $B_d$ of degree $d$, of type I and with $l$ real connected components. Then, by Corollary \ref{cor: odd_even_plane}, the curve $B_d$ admits no totally real pencils of curves of degree $k$ such that $k\leq \frac{l}{d}$.
 \end{prop}
Remark that, in order to prove Proposition \ref{prop: no_pencil_plane_generale}, it is enough to construct real plane separating curves satisfying the hypotheses of Proposition \ref{prop: plane_main}. This is done in Proposition \ref{prop: separanti_piane_harnack}; see Section \ref{subsec: exa}.\\

 The Brill-Noether theorem implies that there exist real algebraic separating curves of genus $g$, which admit no separating morphisms of degree less than $\lfloor \frac{g+3}{2} \rfloor$. 
 On the other hand, Proposition \ref{prop: no_pencil_plane_generale} implies the existence of real separating plane curves $C$ of degree $2s+1$, which admit no separating morphisms of degree less or equal to $l=s^2+n$, where $l$ is the number of connected components of $\mathbb R C$ and $n$ is a non-negative integer such that
 \begin{equation}
    \begin{cases}
      n \equiv 1 \pmod{2}\\
     \frac{2s^2-s+3- \varepsilon}{2} < s^2 + n  < s^2+s-2+ \varepsilon,
    \end{cases}
  \end{equation}
with $\varepsilon \in \{0,1\}$ such that $\varepsilon \equiv s+1 \mod 2$. It follows that, since $s^2+n > \lfloor \frac{g(C)+3}{2} \rfloor =\displaystyle{\frac{2s^2-s+3- \varepsilon}{2}}$, the existence of a real separating curve admitting no separating morphisms of degree less or equal to $s^2+n$ cannot be proved directly via Brill-Noether theory.

\subsection{Statement of the main results}
\label{subsec: main_statement}
 
In this paper, we consider real separating curves embedded in some ambient real surface and focus on their separating morphisms. The key tools of our approach are the fact that all separating curves come equipped with two possible opposite complex orientations and the use of \cite[Theorem 3.2]{Orev21} as shown in \cite[Example 3.3]{Orev21}. 
\begin{thm}(\cite[Theorem 3.2]{Orev21})
\label{thm: stepan}
Let $X$ be a smooth real algebraic surface and $C \subset X$ a non-singular real algebraic separating curve. Let $D$ be a real divisor belonging to the linear system $|C+K_X|$. Assume that $D$ has not $C$ as a component. We may always write $D=2D_0 + D_1$ with $D_1$ a reduced curve and $D_0$ an effective divisor. Let us fix a complex orientation on $\mathbb R C$ and an orientation $\mathfrak O$ on $\mathbb R X \setminus (\mathbb R C \cup \mathbb R D_1)$ which changes each time we cross $\mathbb R C \cup \mathbb R D_1$ at its smooth points. The latter orientation induces a boundary orientation on $\mathbb R C \setminus (\mathbb R C \cap D_1)$. Let $f: C \rightarrow \mathbb C \mathbb P^1$ be a separating morphism. Then it is impossible that, for some $p \in \mathbb R \mathbb P^1$, the set $f^{-1}(p) \setminus \supp(D)$ is non-empty and the two orientations coincide at each point of the set.
\end{thm}

In \cite{Orev21}, Orevkov shows interesting applications of Theorem \ref{thm: stepan}, such as \cite[Theorem 1.1]{Orev21} and the construction of complex schemes (i.e. real schemes endowed with an orientation) in the real projective plane which are realisable by real pseudoholomorphic separating plane curves of odd degree and not realisable by real algebraic separating plane curves of same degree \cite[Proposition 1.5]{Orev21}. Moreover, other applications of \cite[Theorem 3.2]{Orev21} are presented via examples. Now, let us introduce Set-up \ref{set: DODI} and state the main results, Theorems \ref{thm: main} and \ref{thm: main_D1}.

\begin{set}
\label{set: DODI}
 Let $X$ be a smooth real algebraic surface and $C \subset X$ a non-singular real algebraic separating curve. Assume that 
 \begin{enumerate}
 \item[(1)] the surface $X$ has holomorphic Euler characteristic of the trivial bundle $\chi (\mathcal{O}_X) \geq 1$;
	\item[(2)] $(-K_X)^2 \geq 0$;
	\item[(3)]$-K_X.D \geq 0$ for all effective divisors $D$.
 \end{enumerate}
 Let $D$ be a real divisor  belonging to the linear system $|C+K_X|$. Assume that $D$ has not $C$ as a component. Fix a decomposition $2D_0 + D_1$ of the divisor $D$ with the following properties. The divisor $D_0$ is effective and $D_1$ is a reduced curve such that $\frac{D_0^2-D_0.K_X}{2} \in \mathbb Z_{\geq 0}$ is maximised; see Example \ref{exa: even_choiceD1}.
\end{set}

\begin{exa}
\label{exa: even_choiceD1}
	Let $C$ be a non-singular real algebraic curve of degree $2s$ in $\mathbb C \mathbb P^2$. Then $\mathbb C \mathbb P^2$ satisfies (1), (2) and (3) of Set-up \ref{set: DODI} and, if $s \geq 2$, the divisor $D_0$ realises the class of a plane curve of degree $s-2$ and $D_1$ is a line in $\mathbb C \mathbb P^2$.
\end{exa}

\begin{thm}
\label{thm: main}
Assume Setup \ref{set: DODI}. Denote with $l$ the number of connected components of $\mathbb R C$. Then, if
\begin{itemize}
	\item[(4)] $D_1= \emptyset$;
	\item[(5)] $-K_X.D_0 > 0$;
	\item[(6)] 
	$$1-\varepsilon + \frac{D_0^2+ D_0 K_X}{2} < \lfloor \frac{l}{2} \rfloor < \frac{D_0^2-D_0.K_X}{2},$$
	where $\varepsilon \in \{0,1\}$ such that $\varepsilon \equiv l \mod 2$;
\end{itemize}
there are no separating morphisms $f: C \rightarrow \mathbb{C} \mathbb{P}^1$ of degree $l$.
\end{thm}
\begin{rem}
\label{rem: dopo_main_prop_plane}
Theorem \ref{thm: main} implies $(1)$ of Proposition \ref{prop: plane_main}. In fact, the complex projective plane and any non-singular real plane separating curve of odd degree $d \geq 5$, satisfying the bound in $(1)$ of Proposition \ref{prop: plane_main}, satisfy the hypotheses of Theorem \ref{thm: main}. Remark that $D_0$ realises the class of a plane curve of degree $\frac{d-3}{2}$. 
\end{rem}

In the following, we are going to extend Theorem \ref{thm: main} to separating curves in $X$ for which $D_1$ is not empty. In this context, in order to show that there are no separating morphisms $f: C \rightarrow \mathbb{C} \mathbb{P}^1$ of degree $l$, one needs to impose stricter bounds on the number $l$ of connected components of $\mathbb R C$.
\begin{thm}
\label{thm: main_D1}
Assume Setup \ref{set: DODI}. Denote with $l$ the number of connected components of $\mathbb R C$. Assume that $D_1 \neq \emptyset$.
Set $m=\min( \frac{ D_1^2-D_1 K_X}{2}-1, l-1)$. 

Then, if
\begin{itemize}
	\item[(4)] $-K_X.D_1 > 0$;
	\item[(5)] $D_1^2 < l$;
	\item[(6)] $-K_X.D_0 > 0$;
	\item[(7)] 
	$$\max(\lfloor \frac{D_1^2-m}{2} \rfloor, 1-\varepsilon + \frac{D_0^2+ D_0 K_X}{2}) < \lfloor \frac{l-m}{2} \rfloor < \frac{D_0^2- D_0 K_X}{2},$$
		where $\varepsilon \in \{0,1\}$ such that $\varepsilon \equiv l-m \mod 2$,
\end{itemize}
there are no separating morphisms $f: C \rightarrow \mathbb{C} \mathbb{P}^1$ of degree $l$.
\end{thm}
\begin{rem}
\label{rem: maximising_bound}
Remark that the decomposition of the divisor $D$, fixed in Setup \ref{set: DODI}, allows one to maximise the bounds for $l$ in Theorems \ref{thm: main} and \ref{thm: main_D1}.
\end{rem}
As a corollary of Theorems \ref{thm: main} and \ref{thm: main_D1}, we prove Proposition \ref{prop: plane_main}. 
\begin{proof}[Proof of Proposition \ref{prop: plane_main}]
The point $(1)$ of Proposition \ref{prop: plane_main} is proven in Remark \ref{rem: dopo_main_prop_plane}. So, let $C$ be a non-singular real algebraic plane curve as in $(2)$ of Proposition \ref{prop: plane_main}. Then, this curve $C$ satisfies the hypotheses of Theorem \ref{thm: main_D1}. Indeed, the divisor $D_0$ realises the class of a plane curve of degree $s-2$ and $D_1$ is a line. The number $m=\min (1,l-1)$ must be one because Rokhlin's complex orientations formula (\cite{Rokh74}) bounds $l$ to be bigger or equal to $s$. It follows that the hypotheses of Theorem \ref{thm: main_D1} apply to $C \subset \mathbb C \mathbb P^2$.
\end{proof}

\section{Theorems \ref{thm: main} and \ref{thm: main_D1} and applications}
Section \ref{subsec: main} is uniquely devoted to the proof of Theorem \ref{thm: main} and Theorem \ref{thm: main_D1}. Afterwards, in Section \ref{subsec: exa}, we present examples, applications of Theorems \ref{thm: main} and \ref{thm: main_D1}, and we prove Proposition \ref{prop: no_pencil_plane_generale}. 

\subsection{Proof of Theorems \ref{thm: main} and \ref{thm: main_D1}}
\label{subsec: main}
\begin{proof}[Proof of Theorem \ref{thm: main}]	Denote by $C_i$ the connected components of $\mathbb R C$ where $i=1,\dots l$. For the sake of contradiction, assume that there exists a separating morphism $f: C \rightarrow \mathbb P^1$ of degree $l$. Then, for every $p \in \mathbb R\mathbb P^1$, the set $f^{-1}(p)=\{p_1,\dots,p_{l}\}$ is a collection of $l$ real points such that every $C_i$ contains exactly one $p_i$. 
The positive integer $\frac{D_0^2-D_0.K_X}{2}$ is always less or equal than the dimension of the linear system of curves in $X$ of class $D_0$. To prove the latter is enough to use Riemann-Roch theorem
 $$ \dim H^0(X,\mathcal L (D_0)) - \dim H^1(X,\mathcal L (D_0))+ \dim H^2(X,\mathcal L (D_0))= \chi (\mathcal O_{X})+ \frac{D_0^2-D_0.K_X}{2}$$
 and the hypotheses of Theorem \ref{thm: main}. In fact, hypothesis (1) implies that
 $$\dim H^0(X,\mathcal L (D_0)) + \dim H^2(X,\mathcal L (D_0)) \geq 1 + \frac{D_0^2-D_0.K_X}{2};$$ hypotheses (2), (5) imply that  $-K_X.(K_X-D_0)<0$. Then, hypothesis (3) implies that $K_X-D_0$ is non-effective and, therefore, the dimension of $H^0(X,\mathcal L (K_X-D_0))$ is zero. Finally $H^2(X,\mathcal L (D_0)) \simeq \dim H^0(X,\mathcal L (K_X-D_0))$ by Serre duality.

It follows that, fixed a configuration of $\frac{D_0^2-D_0.K_X}{2}$ points, there always exists at least one curve of class $D_0$ passing through the configuration.

Let us fix an orientation on $\mathbb R X  \setminus \mathbb R C$ which changes each time we cross $\mathbb R C$. The latter orientation induces a boundary orientation $\mathfrak O$ on $\mathbb R C$. Fix some $p$ in $\mathbb R \mathbb P^1$. 	

 Define $Y_p$ (resp. $NY_p$) as the set of points in $f^{-1}(p) \setminus \supp(D)$ on which the complex orientation and the orientation $\mathfrak O$ agree (resp. do not agree). Remark that Theorem \ref{thm: stepan} tells us that if the set $f^{-1}(p) \setminus \supp(D)$ is non-empty, one has that $Y_p$ and $NY_p$ are both non-empty. In the following, we are going to use Theorem \ref{thm: stepan} to get a contradiction.

Fix one of the two complex orientations on $\mathbb R C$. Then, remark that either the number $V^1_p$ of $p_i$'s where the complex orientation and the orientation $\mathfrak O$ agree, or the number $V^2_p$ of the $p_i$'s where the two orientations do not agree, does not exceed $\lfloor \frac{l}{2} \rfloor$. Given a collection $\mathcal P$ of $\frac{D_0^2-D_0.K_X}{2}$ real points in $X$, there exists a curve $A$ of class $D_0$ passing through $\mathcal P$. By hypothesis $ \lfloor \frac{l}{2} \rfloor < \frac{D_0^2-D_0.K_X}{2}$, therefore it is possible to choose such a collection $\mathcal P$ so that it contains at least $ \lfloor \frac{l}{2} \rfloor$ points of $f^{-1}(p)$ belonging to $V^j_p$, for an opportune $j \in \{1,2\}$. It follows that, if the set $f^{-1}(p) \setminus \supp(D)$ is non-empty, either $NY_p$ or $Y_p$ is empty, which contradicts Theorem \ref{thm: stepan}. 
It remains to show that $f^{-1}(p) \setminus \supp(D)$ is non-empty. Assume that for all configurations $\mathcal P$ of 
$$\frac{D_0^2-D_0.K_X}{2}= \lfloor \frac{l}{2} \rfloor + h$$ 
points, containing at least $ \lfloor \frac{l}{2} \rfloor$ points of $f^{-1}(p)$ belonging to $V^j_p$, for an opportune $j \in \{1,2\}$, every curve $A$ of class $D_0$ passing through $\mathcal P$ contains $f^{-1}(p)$. Since $h > 0$, it follows that there exist at least two distinct curves $A, \tilde A$ of class $D_0$ intersecting in at least $l+h-1$ points. On the other hand 
$$D_0^2  < l+h-1= \frac{D_0^2-D_0.K_X}{2} + \lfloor \frac{l}{2} \rfloor + \varepsilon -1 $$ because of hypothesis $(6)$ and, therefore, we obtain a contradiction.

\end{proof}
Let us prove Theorem \ref{thm: main_D1} via a strategy analogous to that explained in the proof of Theorem \ref{thm: main}, minding the role played by $D_1$; see Setup \ref{set: DODI}.
\begin{proof}[Proof of Theorem \ref{thm: main_D1}]
Denote by $C_i$ the connected components of $\mathbb R C$ where $i=1,\dots l$. For the sake of contradiction, assume that there exists a separating morphism $f: C \rightarrow \mathbb P^1$ of degree $l$. 
Let us fix an orientation on $\mathbb R X  \setminus (\mathbb R C \cup \mathbb R D_1)$ which changes each time we cross $\mathbb R C \cup \mathbb R D_1$ at its smooth points; see \cite[Proof of Theorem 3.2]{Orev21} for details. The latter orientation induces a boundary orientation $\mathfrak O$ on $\mathbb R C$. Moreover, fix some $p$ in $\mathbb R \mathbb P^1$ and denote with $p_i \in C_i$ the points belonging to $f^{-1}(p)$. Analogously to the proof of Theorem \ref{thm: main}, thanks to hypotheses (1)-(4) and Riemann-Roch theorem, one can pick $D_1$ as the curve passing through a given collection $\tilde{\mathcal P}$ of $\frac{D_1^2-D_1 K_X}{2}$ points  contained in $m= \min (\frac{D_1^2-D_1 K_X}{2}-1,l-1)$ connected components $C_{j_1}, \dots,C_{j_m}$. Moreover, choose the collection $\tilde{\mathcal P}$ such that it contains the points $p_{j_1},\dots,p_{j_m}$ of $f^{-1}(p)$. Remark that $D_1$ cannot contain all points of $f^{-1}(p)$ because of hypothesis $(5)$. It may happen that $D_1$ contains more than $m$ points among those of $f^{-1}(p)$, nevertheless, in general, we do not have control on that.

As in the proof of Theorem \ref{thm: main}, define $Y_p$ and $NY_p$. Fix one of the two complex orientations on $\mathbb R C$. Then, remark that either the number $V^1_p$ of $p_i$'s contained in $\mathbb R C \setminus \bigsqcup_{i=1}^{m} C_{j_i}$, where the complex orientation and the orientation $\mathfrak O$ agree, or the number $V^2_p$ of the $p_i$'s in $\mathbb R C \setminus \bigsqcup_{i=1}^{m} C_{j_i}$, where the two orientations do not agree, does not exceed $\lfloor \frac{l-m}{2} \rfloor $. Therefore, we pick a collection $\mathcal P$ of $\frac{D_0^2-D_0.K_X}{2}$ real points such that these include the points of $V^j_p$, for an opportune $j \in \{1,2\}$. Once again, thanks to the same argument used in the proof of Theorem \ref{thm: main}, the set $f^{-1}(p) \setminus \supp(D)$ is non-empty,  and either $NY_p$ or $Y_p$ is empty, which contradicts Theorem \ref{thm: stepan}.
\end{proof}
\subsection{Examples and Applications}
\label{subsec: exa} First, let us introduce some notation and terminology for real curves in the real projective plane. The real locus of a non-singular real plane curve is homeomorphic to a disjoint union of circles embedded in $\mathbb{R}\mathbb P^2$. Each circle can be embedded in $\mathbb{R}\mathbb P^2$ in two different ways: if it realises the trivial-class in $H_1(\mathbb{R} \mathbb P^2; \mathbb{Z}/2\mathbb{Z})$, it is called \textit{oval}, otherwise it is called \textit{pseudoline}. If a non-singular real plane curve has even degree then its real locus consists of ovals only; otherwise of exactly one pseudoline and ovals.

An oval in $\mathbb{R}\mathbb{P}^{2}$ separates two disjoint non-homeomorphic connected components: the connected component homeomorphic to a disk is called \textit{interior} of the oval; the other one is called \textit{exterior} of the oval. For each pair of ovals, if one is in the interior of the other we speak about an \textit{injective pair}, otherwise a \textit{non-injective pair}. 

\begin{defn}
\label{defn: real_scheme}

 Let $A \subset \mathbb C \mathbb P^2$ be a non-singular real algebraic curve. We say that $A$ has real scheme $\mathcal{S}$ if the pair $(\mathbb{R} \mathbb P^2, \mathbb{R}A)$ realises the topological type $\mathcal{S}$, up to homeomorphism of $\mathbb R \mathbb P^2$.

\end{defn}

Let us consider real separating curves in the complex projective plane and focus on \cite[Question 3.6]{KumShaw20} (see the beginning of Section \ref{subsec: obstruction_pencil}). A first known example concerns real separating plane curves of degree $2s$ or $2s+1$  having a \textit{nest} of maximal depth $s$, i.e. there are $s$ ovals and any two ovals of the collection form an injective pair. These curves admit a totally real pencil of lines. In fact, for any fixed point $q$ in the interior of the innermost oval of the nest, there exists a totally real pencil of lines with base point $q$. Some more examples can be found in \cite{Touz13}, where for some pairs $(A, \mathcal S)$, where $A$ is a real separating plane sextic with $9$ ovals and $\mathcal S$ its real scheme in $\mathbb R \mathbb P^2$, the minimal value for $k$ of Theorem \ref{thm: KS}, is shown to be equal to $3$. 
\begin{figure}[h!]
\begin{picture}(100,109)

\put(-25,-10){\includegraphics[width=0.25\textwidth]{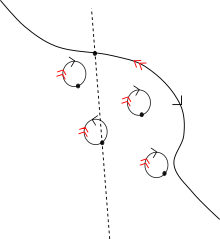}}
\end{picture}
\caption{$(\mathbb R \mathbb P^2, \mathbb R A, \mathbb R L)$ of Example \ref{exa: quintic}. Double arrows denote $\mathfrak O$, simple arrows the fixed complex orientation of $\mathbb R A$ and $\bullet$ the points in $f^{-1}(p)$.}
\label{fig: quintic}
\end{figure} 
\begin{exa}

\label{exa: quintic} Let $A$ be a non-singular plane quintic of type I realising the real scheme with $4$ ovals; see Fig. \ref{fig: quintic}. 
We are going to show that there exist no separating morphisms $f: A \rightarrow \mathbb C \mathbb P^1$ of degree $5$. This result does not follow directly from Proposition \ref{prop: plane_main}, but one can still use Theorem \ref{thm: stepan} in order to prove the statement.

For the sake of contradiction assume that such an $f$ exists. Fix $p \in \mathbb R \mathbb P^1$. Let us use the notation in Setup \ref{set: DODI}. The divisor $D_0$ realises the class of a line. Fix an orientation $\mathfrak O$ that changes each time we cross $\mathbb R A$. Moreover, choose one of the two complex orientations of $\mathbb R A$. Let $a$ (resp. $b$) be the number of connected components of $\mathbb R A$
where the two orientations coincide (resp. are opposite). Since $a+b=5$, we have $a\le 2$ or $b\le 2$. Up to exchange $a$ and $b$
(by reversing the complex orientation), we assume that $b\le 2$ and we trace a line $L$ through the points in $f^{-1}(p)$ belonging to these $b$ connected components; see as a toy example Fig. \ref{fig: quintic}. Since $L$ cannot contain all points in $f^{-1}(p)$, we get a contradiction thanks to Theorem \ref{thm: stepan}. 
\end{exa}

In the following, we prove Proposition \ref{prop: separanti_piane_harnack}, which implies Proposition \ref{prop: no_pencil_plane_generale}. The content of Proposition \ref{prop: separanti_piane_harnack} is a well known fact in the study of the topology of real algebraic separating plane curves. Nevertheless, to our knowledge, there is no proof of it in the literature. Therefore, for the interested reader, we present a proof relying on a variation of Harnack's construction method \cite{Harn76}, \cite[pag. II.4]{Mari79} and \cite[Section 2]{FIed81}, which allows us to obtain type I curves when perturbing the union of two type I curves intersecting transversally in real points only.
\begin{thm}{\cite[Section 2]{FIed81}}
\label{thm: Fiedler_satz}
	Let $A_1$ and $A_2$ be two non-singular real plane projective separating curves of degree respectively $d_1$ and $d_2$ such that they intersect transversally in $d_1d_2$ distinct real double points and each of them is equipped with one of its two complex orientations. Let $B$ be a non-singular real plane curve of degree $d=d_1 + d_2$ obtained from $A_1 \cup A_2$ by a small perturbation respecting the chosen complex orientations at every smoothing of a double point. Then $B$ is separating and the orientation obtained on $\mathbb R B$ from those of $\mathbb R A_1 \cup \mathbb R A_2$ is a complex orientation of $B$.
\end{thm}
\begin{prop}
	\label{prop: separanti_piane_harnack}
	For every positive integer $d$ and for every $\lceil \frac{d}{2} \rceil \leq l_d \leq \frac{(d-1)(d-2)}{2}+1$, where $l_d$ is an integer that has the parity of $\lceil \frac{d}{2} \rceil$, there exists a non-singular real separating plane projective curve $A$ of degree $d$ such that $\mathbb R A$ has $l_d$ connected components.
\end{prop}
\begin{proof}
In the following, whenever we consider a plane curve $A$, we denote its polynomial by $a(x,y,z)$. Let us fix a real line $L$ in $\mathbb C \mathbb P^2$. In order to prove the statement, we show that for any positive integer $d$ and for any integer $l_{d}$ satisfying the hypotheses, there exists a non-singular real plane separating curve $B_{d}$ with $l_{d}$ real connected components such that a unique connected component of $\mathbb R B_{d}$ intersects $\mathbb R L$ transversally in $d$ points as in Fig. \ref{fig: simple}.

\textbf{Base case}: 
For fixed degree $d=1,2$ and $3$ the only possible value for $l_d$ is respectively $1$ and $2$ and any non-singular real curve $B_d$ of one such degree $d$ and such $l_d$ as number of real connected components is maximal, therefore it can be obtained directly via Harnack construction's method (\cite{Harn76}); in particular, the curve $B_d$ can be constructed such that a unique connected component of $\mathbb R B_d$ intersects $\mathbb R L$ in $d$ points.

 \begin{figure}[h!]
\begin{picture}(100,110)
\put(-120,35){\includegraphics[width=0.7\textwidth]{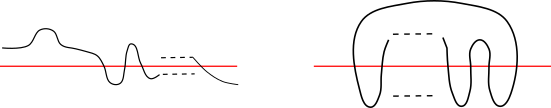}}

\put(-63,30){\rotatebox{-180}{$\overbrace{ \quad \quad \quad \quad \quad \quad }$ }}
\put(-75,15){$d \text{ real intersections}$}
\put(-60,0){Fig. \ref{fig: simple}.1 }

\put(110,30){\rotatebox{-180}{$\overbrace{\quad \quad \quad \quad \quad \quad \quad \quad }$}}
\put(104,15){$d \text{ real intersections}$}
\put(135,0){Fig. \ref{fig: simple}.2 }

\end{picture}
\caption{The union of a connected component of $\mathbb R B_{d}$ and $\textcolor{red}{ \mathbb R L}$. The degree of $B_{d}$ is odd in Fig. \ref{fig: simple}.1 and even in Fig. \ref{fig: simple}.2}
\label{fig: simple}
\end{figure} 

 \begin{figure}[h!]
\begin{picture}(100,230)
\put(-120,15){\includegraphics[width=0.7\textwidth]{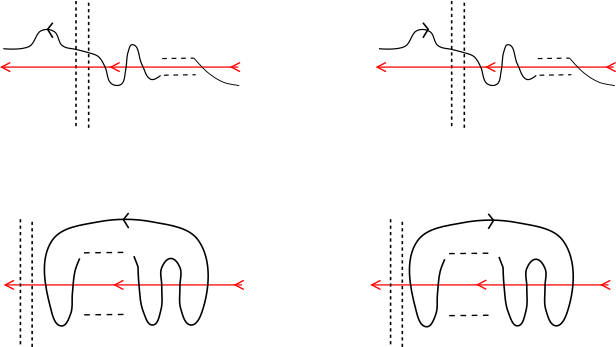}}
\put(-96,150){\textcolor{red}{$\mathcal H$}}
\put(107,150){\textcolor{red}{$\mathcal H$}}
\put(-123,65){\textcolor{red}{$\mathcal H$}}
\put(75,65){\textcolor{red}{$\mathcal H$}}
\put(-64,120){Fig. \ref{fig: opposite_perturbation}.1 }
\put(-84,210){$\overbrace{ \quad }^{d \text{ lines}}$ }
\put(135,120){Fig. \ref{fig: opposite_perturbation}.2 }
\put(119,210){$\overbrace{ \quad }^{d \text{ lines}}$ }
\put(-64,-2){Fig. \ref{fig: opposite_perturbation}.3 }
\put(-116,90){$\overbrace{ \quad }^{d \text{ lines}}$ }
\put(135,-2){Fig. \ref{fig: opposite_perturbation}.4 }
\put(87,90){$\overbrace{ \quad }^{d \text{ lines}}$ }

\end{picture}
\caption{The union of a connected component of $\mathbb R B_{d-1}$, the line $\textcolor{red}{ \mathbb R L}$ and $d$ real lines. The degree of $B_{d-1}$ is odd in Fig. \ref{fig: opposite_perturbation}.1-2 and even in Fig. \ref{fig: opposite_perturbation}.3-4. Arrows denote a fixed complex orientation on respectively $\mathbb R B_{d-1}$ and $\textcolor{red}{ \mathbb R L}$.}
\label{fig: opposite_perturbation}
\end{figure} 

	\textbf{Induction step}: Assume that for degree $d-1$ and for any integer $l_{d-1}$ having the parity of $\lceil \frac{d-1}{2} \rceil$ and  with $\lceil \frac{d-1}{2} \rceil \leq l_{d-1} \leq \frac{(d-2)(d-3)}{2}+1$, there exists a non-singular real plane separating curve $B_{d-1}$ with $l_{d-1}$ real connected components such that a unique connected component of $\mathbb R B_{d-1}$ intersects $\mathbb R L$ transversally in $d-1$ points as in Fig. \ref{fig: simple}.
	
	For any fixed pair $(d-1, l_{d-1})$ as above, we are going to construct a non-singular real plane separating curve $B_{d}$ of degree $d$
 \begin{enumerate}
 	\item[(i)]  with $l_{d}=l_{d-1}+d-2$
 	\item[(ii)] with $l_{d}=l_{d-1}+ \delta$, where $\delta=d \mod 2$ and $\delta \in \{0,1\}$,
 \end{enumerate}
 real connected components such that a unique connected component of $\mathbb R B_{d}$ intersects $\mathbb R L$ in $d$ points as in Fig. \ref{fig: simple}.  
 Such construction will end the proof. In fact, remark that, for any given pair $(d,l_d)$ such that $l_d \equiv \lceil \frac{d}{2} \rceil \mod 2$ and
 $$ \lceil \frac{d}{2} \rceil \leq l_d \leq \frac{(d-1)(d-2)}{2}+1,$$ 
 there exists, by induction hypothesis, a non-singular real plane separating curve of degree $d-1$ either with $l_{d-1}=l_d -d+2 $ or $l_{d-1}=l_d- \delta $.

Let us fix a complex orientation on $\mathbb R L$ and $\mathbb R B_{d-1}$:
\begin{itemize}
	\item  as in Fig. \ref{fig: opposite_perturbation}.1, respectively Fig. \ref{fig: opposite_perturbation}.2, if $d-1$ is odd.
	\item as in Fig. \ref{fig: opposite_perturbation}.3, respectively Fig. \ref{fig: opposite_perturbation}.4, if $d-1$ is even.
\end{itemize}
Moreover, choose a connected component $\mathcal H$ among those of $\mathbb R L \setminus \mathbb R B_{d-1}$. If $d-1$ is even, choose $\mathcal H$ in the exterior of the oval intersecting $\mathbb R L$; see Fig. \ref{fig: opposite_perturbation}. 
 \begin{figure}[h!]
\begin{picture}(100,230)
\put(-120,15){\includegraphics[width=0.7\textwidth]{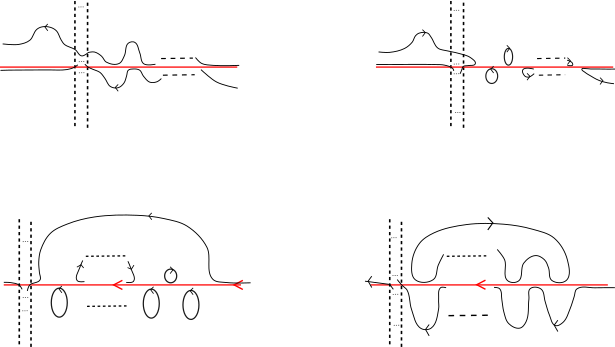}}
\put(-95,120){Fig. \ref{fig: perturbed}.1: $l_d=l_{d-1}$ }
\put(-97,210){$\overbrace{ \quad }^{d \text{ intersections}}$ }
\put(93,120){Fig. \ref{fig: perturbed}.2: $l_d=l_{d-1}+d-2$ }
\put(107,210){$\overbrace{ \quad }^{d \text{ intersections}}$ }
\put(-105,-2){Fig. \ref{fig: perturbed}.3: $l_d=l_{d-1}+d-2$ }
\put(-128,90){$\overbrace{ \quad }^{d \text{ intersections}}$ }
\put(105,-2){Fig. \ref{fig: perturbed}.4: $l_d=l_{d-1}+1$ }
\put(73,90){$\overbrace{ \quad }^{d \text{ intersections}}$ }
\end{picture}
\caption{A small type I perturbation of the union of a connected component of $\mathbb R B_{d-1}$, the line $\textcolor{red}{ \mathbb R L}$ and $d$ real lines. The degree of $B_{d}$ is even in Fig. \ref{fig: perturbed}.1-2 and odd in Fig. \ref{fig: perturbed}.3-4.}
\label{fig: perturbed}
\end{figure} 

Pick $d$ real lines $A_1,\dots, A_d$ intersecting transversally $\mathcal H$. Then, for $\varepsilon \in \mathbb R_{\not = 0}$ small enough, up to a choice of the sign of $\varepsilon$, one can perturb $B_{d-1}\cup L$ and obtain a real separating curve $B_d$ as zero set of the polynomial $b_{d-1}(x,y,z)l(x,y,z)+\varepsilon a_1(x,y,z)\dots a_d(x,y,z)=0$ and such that

\begin{itemize}
	\item a connected component $\mathbb R B_{d}$ intersects $\mathbb R L$ transversally in $d$ points as in Fig. \ref{fig: perturbed};
	\item  with $l_{d}=l_{d-1}+d-2$, respectively with $l_{d}=l_{d-1}+ \delta$, where $\delta=d \mod 2$ and $\delta \in \{0,1\}$, real connected components; see Fig. \ref{fig: perturbed}.
\end{itemize}

\end{proof}

We conclude looking at separating real curves in a different ambient surface.

Let $Q$ be $\mathbb{C} \mathbb P^1 \times \mathbb{C} \mathbb P^1$, equipped with the anti-holomorphic involution $\sigma: Q \rightarrow Q$ sending $(x,y)$ to $(\overline y, \overline x)$, where $x=[x_0:x_1]$ and $y=[y_0:y_1]$ are in $\mathbb{C} \mathbb P^1$ and $\overline{x}=[\overline{x_0}:\overline{x_1}]$ and $\overline{y}=[\overline{y_0}:\overline{y_1}]$ are respectively the images of $x$ and $y$ via the standard complex conjugation on $\mathbb{C} \mathbb P^1$. The real part of $Q$ is homeomorphic to a $2$-sphere and $Q$ is called \textit{quadric ellipsoid}. A non-singular real algebraic curve $A$ on $Q$ is defined by a bi-homogeneous polynomial of bidegree $(d, d)$ 
$$P(x,y)= \sum \limits_{0 \leq i,j \leq d} a_{i,j}x_0^ix_1^{d-i}y_0^jy_1^{d-j}=0,$$ 
where $d$ is a positive integer and the coefficients satisfy $a_{i,j}=\overline{a_{j,i}}$. If $A$ is separating then the number of the connected components of $\mathbb R A$ has the parity of $d$.

\begin{cor}

	\label{cor: quadric_ellipsoid}
	Let $A$ be a non-singular real algebraic curve of type I and with $l$ real connected components in the quadric ellipsoid $Q$. Assume one of the following:
\begin{enumerate}
	
	\item The bidegree of $A$ is $(2s,2s)$, for some $s \in \mathbb Z_{\geq2}$ and $$ 1 + s^2-4s+3 < \frac{l}{2}  < s^2-1.$$
	\item The bidegree of $A$ is $(2s+1, 2s+1)$, for some $s \in \mathbb Z_{\geq 2}$ and 
	
	$$ \max (0, s^2-4s+3 ) <  \frac{l-3}{2} < s^2-1.$$ 
\end{enumerate}
Then $A$ admits no separating morphisms of degree $l$.
\end{cor}
\begin{proof}
	Theorems \ref{thm: main_D1} and \ref{thm: main} apply to real separating bidegree $(d,d)$ curves on the quadric ellipsoid. If $d=2s+1$, the curve $D_1$ has bidegree $(1,1)$ and, therefore, the number $m$ equals $2$; otherwise, i.e. if $d=2s$, the curve $D_1$ is empty. The divisor $D_0$ realises the class of a curve of bidegree $(s-1,s-1)$.
\end{proof}
Real schemes realised by non-singular real algebraic curves of bidegree $(d,d)$ on $Q$ are completely classified for all $d \leq 5$; see \cite{GudShu80}, \cite{Mikh94} and \cite{Manz21}. Moreover such classifications distinguish the cases in which a given topological type may or may not be realised by a real algebraic curve of type I. 
 
For $d \leq 5$, we report in Table \ref{tabellalastchapter} a complete list of real schemes realised by type I real algebraic curves of bidegree $(d,d)$, with $l$ real connected components, where $l$ satisfies the hypotheses of Corollary \ref{cor: quadric_ellipsoid}. Hence, real separating curves of bidegree $(d,d)$ realising any among such real schemes do not admit separating morphisms of degree $l$. In order to understand Table \ref{tabellalastchapter}, let us introduce some notation. Let $A$ be a non-singular bidegree $(d,d)$ real curve on $Q$. The real connected components of $A$ are called \textit{ovals}. An oval in $\mathbb R Q$ bounds two disks; therefore, on $\mathbb R Q$ interior and exterior of an oval are not well defined. It follows that the encoding of real schemes on $\mathbb R Q$ is not well defined either and it depends on the choice of a point on $\mathbb R Q \setminus \mathbb R A$. Let $\bigsqcup_{i}B_i$ be a collection of ovals in $\mathbb R Q$. We say that the pair $(\mathbb R Q,\bigsqcup_{i}B_i )$ realises $\mathcal S$ if there exists a point $p \in \mathbb R Q \setminus \bigsqcup_{i}B_i $ such that $(\mathbb R Q \setminus \{p\},\bigsqcup_{i}B_i )$ realises $\mathcal S$. So that, in order to encode the topology of a real scheme in $\mathbb R Q$, we introduce some notation to encode that of real schemes in $\mathbb{R}^{2}$, which is homeomorphic to $\mathbb R Q$ deprived of a point. Let us call \textit{oval} any circle embedded in $\mathbb{R}^2$. Analogously to the case of $\mathbb{R} \mathbb P^{2}$, in $\mathbb{R}^{2}$ one can define interior and exterior of an oval and (non-)injective pairs for each pair of ovals; see the beginning of Section \ref{subsec: exa}. 
\begin{note}
	Let us consider collections of disjoint ovals in $\mathbb{R}^{2}$. An empty collection of ovals is denoted by $\langle 0 \rangle $. We say that a disjoint collection of $l$ ovals realises $\langle l \rangle$ if there are no injective pairs. The symbol $\langle 1\langle \mathcal{S} \rangle \rangle$ denotes the disjoint union of a non-empty collection of ovals realising $\langle \mathcal{S} \rangle$, and an oval forming an injective pair with each oval of the collection. The disjoint union of any two collections of ovals, realising respectively $\langle \mathcal{S}' \rangle$ and $\langle \mathcal{S}''\rangle$ in $\mathbb R^2$, is denoted by $\langle \mathcal{S}'  \sqcup  \mathcal{S}'' \rangle$ if none of the ovals of one collection forms an injective pair with the ovals of the other one and they are both non-empty collections. 
\end{note}
	\begin{table}[h!]
\centering
\begin{tabular}{ |l| l | r | c |l |c |c |}
\hline
d&Real schemes realised by type I curves& $l(D_0,D_1)$&$ \frac{l-j}{2} $& $r(D_0,D_1)$&$D_0$&$D_1$ \\

\hline
$4$&$\langle 1\langle 1\langle 1\langle 1 \rangle \rangle \rangle \rangle$&0&$2$&$3$&$(1,1)$&$-$\\
\hline
$5$&$\langle 1\langle  1\langle 1\langle 1\langle 1 \rangle \rangle \rangle \rangle \rangle$&$0$&$1$&$3$&$(1,1)$&$(1,1)$\\
\hline
$5$&$\langle \alpha \quad \sqcup \quad 1\langle \beta \rangle  \quad \sqcup \quad 1\langle \gamma \rangle \rangle,$ for all $\alpha, \beta, \gamma$ such&$0$&$2$&$3$&$(1,1)$&$(1,1)$ \\ 
& that $\alpha =0 \pmod 2$ and $\alpha + \beta + \gamma=5$ & &&&& \\
\hline
\end{tabular}
\caption{\label{tabellalastchapter} For each $d \leq 5$, this is a complete list of real schemes realised by type I real algebraic curves of bidegree $(d,d)$, with $l$ ovals, where $l$ satisfies the hypotheses of Corollary \ref{cor: quadric_ellipsoid}. The symbols $l(D_0,D_1),$ $r(D_0,D_1)$ represent the bounds respectively on the left and on the right for the number $\frac{l-j}{2}$ in $(1)-(2)$ of Corollary \ref{cor: quadric_ellipsoid}, where the index $j$ equals $0$ if $d$ is even and $3$ otherwise.}
\end{table}

Moreover, one can construct, in a way similar to the proof of Proposition \ref{prop: separanti_piane_harnack}, for every $d \geq 4$, bidegree $(d,d)$ separating real algebraic curves on $Q$ not admitting separating morphism of degree equal to the number of their real connected components. In fact, variations of Harnack's construction method and Rokhlin's complex orientations formula \cite{Zvon91}, \cite[Section 1.2]{Orev207} are also available on $Q$. 
\begin{rem}
In general, other real ambient surfaces as the quadric hyperboloid, del Pezzo surfaces, Hirzebruch surfaces have been previously studied and many real topological tools are developed to investigate applications of Theorem \ref{thm: stepan}; in order to have an idea of these settings see e.g. \cite{DegKha00}, \cite{GudShu80}, \cite{Mikh98}, \cite{Manz22}, \cite{Orev03}.
	\end{rem}
\bibliographystyle{alpha}
\bibliography{biblio.bib}

\begin{thebibliography}{KTM23}

\bibitem[Ahl50]{Ahlf50}
L.~Ahlfors.
\newblock Open {R}iemann surfaces and extremal problems on compact subregions.
\newblock {\em Comment. Math. Helv.}, 24:100--134, 1950.

\bibitem[CH13]{CopHui13}
M.~Coppens and J.~Huisman.
\newblock Pencils on real curves.
\newblock {\em Math. Nachr.}, 286(8-9):799--816, 2013.

\bibitem[Cop13]{Copp13}
M.~Coppens.
\newblock The separating gonality of a separating real curve.
\newblock {\em Monatsh. Math.}, 170:1--10, 2013.

\bibitem[DK00]{DegKha00}
A.~I. Degtyarev and V.~M. Kharlamov.
\newblock Topological properties of real algebraic varieties: {R}okhlin's way.
\newblock {\em Russ. Math. Surv.}, 55(4):735--814, 2000.

\bibitem[Fie81]{FIed81}
T.~Fiedler.
\newblock {Eine Beschr\"ankung f\"ur die Lage von reellen ebenen algebraischen
  Kurven}.
\newblock {\em Beitr\"age zur Algebra und Geometrie}, 11:7--19, 1981.

\bibitem[Gab06]{Gaba06}
A.~Gabard.
\newblock Sur la repr{\'e}sentation conforme des surfaces de {R}iemann {\`a}
  bord et une caract{\'e}risation des courbes s{\'e}parantes.
\newblock {\em Comment. Math. Helv.}, 81(4):945--964, 2006.

\bibitem[GS80]{GudShu80}
D.~A. Gudkov and E.~I. Shustin.
\newblock {Classification of non-singular 8th order curves on ellipsoid}.
\newblock In {\em Methods of the qualitative theory of differential equations},
  pages 104--107(Russian). Gor'kov. Gos. Univ., Gorki, 1980.

\bibitem[Har76]{Harn76}
A.~Harnack.
\newblock {\"{U}ber vieltheiligkeit der ebenen algebraischen Curven}.
\newblock In {\em Math. Ann., 10: 189-199,}, volume 1060 of {\em Lecture Notes
  in Math.}, pages 187--200. 1876.

\bibitem[Hui01]{Huis01}
J.~Huisman.
\newblock On the geometry of algebraic curves having many real components.
\newblock {\em Rev. Mat. Complut.}, 14(1):83--92, 2001.

\bibitem[Kle73]{Klei73}
F.~Klein.
\newblock {{\"U}ber Fl{\"a}chen dritter Ordnung}.
\newblock {\em Math. Ann. 6, pag. 551--581}, 1873.

\bibitem[KS20a]{KumSha20}
M.~Kummer and E.~Shamovich.
\newblock {Real fibered morphisms and Ulrich sheaves}.
\newblock {\em J. Algebraic Geom.}, 29(1):167 -- 198, 2020.

\bibitem[KS20b]{KumShaw20}
M.~Kummer and K.~Shaw.
\newblock The separating semigroup of a real curve.
\newblock {\em Annales de la Facult{\'e} des sciences de Toulouse:
  Math{\'e}matiques}, 29(1):79--96, 2020.

\bibitem[KTM23]{KumLeTMan22}
M.~Kummer, C.~Le Texier, and M.~Manzaroli.
\newblock {Real-Fibered Morphisms of del Pezzo Surfaces and Conic Bundles}.
\newblock {\em Discrete \& Computational Geometry}, 69:849--872, 2023.

\bibitem[Man21]{Manz21}
M.~Manzaroli.
\newblock Real algebraic curves of bidegree (5,5) on the quadric ellipsoid.
\newblock {\em St. Petersburg Math. J.}, 32(2):279--306, 2021.

\bibitem[Man22]{Manz22}
M.~Manzaroli.
\newblock {R}eal algebraic curves on real del {P}ezzo surfaces.
\newblock {\em International Mathematics Research Notices}, 2022(2):1350--1413,
  2022.

\bibitem[Mar79]{Mari79}
A.~Marin.
\newblock {La transversalit{\'e} topologique. Une extension d'un
  th{\'e}or{\`e}me de Rohlin et application au 16e probl{\`e}me de Hilbert}.
\newblock Th{\`e}se de l'Universit{\'e} Paris-Sud, 1979.

\bibitem[Mik94]{Mikh94}
G.~Mikhalkin.
\newblock Congruences for real algebraic curves on an ellipsoid.
\newblock {\em Advances in Soviet Mathematics}, 18:223--233, 1994.

\bibitem[Mik98]{Mikh98}
G.~Mikhalkin.
\newblock Topology of curves of degree 6 on cubic surfaces in
  $\mathbb{R}{P}^3$.
\newblock {\em J. Algebraic Geom.}, 7(2):219--237, 1998.

\bibitem[Mis75]{Mish75}
N.~M. Mishachev.
\newblock Complex orientations of plane {M}-curves of odd degree.
\newblock {\em Funct Anal Its Appl}, 9:342--343, 1975.

\bibitem[Ore03]{Orev03}
S.~Y. Orevkov.
\newblock Riemann existence theorem and construction of real algebraic curves.
\newblock {\em Ann. Fac. Sci. Toulouse Math. (6)}, 12(4):517--531, 2003.

\bibitem[Ore08]{Orev207}
S.~Y. Orevkov.
\newblock Arrangements of an {$M$}-quintic with respect to a conic that
  maximally intersect the odd branch of the quintic.
\newblock {\em St. Petersbourg Math. J.}, 19(4):625--674, 2008.

\bibitem[Ore18]{Orev217}
S.~Y. Orevkov.
\newblock On the hyperbolicity locus of a real curve.
\newblock {\em Funct. Anal. and its Applications}, 52:151--153, 2018.

\bibitem[Ore21]{Orev21}
S.~Y. Orevkov.
\newblock Algebraically unrealizable complex orientations of plane real
  pseudoholomorphic curves.
\newblock {\em Geom. and Funct. Anal.}, (31):930--947, 2021.

\bibitem[Rok74]{Rokh74}
V.~A. Rokhlin.
\newblock Complex orientations of real algebraic curves.
\newblock {\em Funct Anal Its Appl}, 8:331--334, 1974.

\bibitem[Tou13]{Touz13}
S.~Fiedler-Le Touz{\'e}.
\newblock Totally real pencils of cubics with respect to sextics.
\newblock {\em arXiv:1303.4341}, 2013.

\bibitem[Zvo92]{Zvon91}
V.~I. Zvonilov.
\newblock Complex topological invariants of real algebraic curves on a
  hyperboloid and an ellipsoid.
\newblock {\em St. Petersbg. Math. J.}, 3(5):1023--1042, 1992.

\end{thebibliography}

\par\nopagebreak
Matilde Manzaroli, \textsc{Universit\"{a}t T\"{u}bingen, Germany}\par\nopagebreak
 \textit{E-mail address}: \texttt{ matilde.manzaroli'at'uni-tuebingen.de}
\end{document}